\documentclass[12pt]{amsart}
\usepackage{amsmath, amsthm,amssymb,color}

\definecolor{aablue}{rgb}{0.09,0.32,0.44} 
\usepackage[pdfborder={0 0 0}, pdfborderstyle={}, pagebackref, colorlinks, linkcolor=black, citecolor=aablue, urlcolor=blue, hypertexnames=true]{hyperref}

\usepackage{setspace}
\usepackage{amsrefs}

\allowdisplaybreaks

\title{Divisibility and Laws in Finite Simple Groups}

\author{Gady Kozma}
\address{G.K., Department of Mathematics, The Weizmann Institute of Science, POB 26, Rehovot 76100, Israel}
\email{gady.kozma@weizmann.ac.il}

\author{Andreas Thom}
\address{A.T., Mathematisches Institut, U Leipzig,
PF 100920, 04009 Leipzig, Germany}
\email{andreas.thom@math.uni-leipzig.de}

\date{\today}

\newtheorem{theorem}{Theorem}[section]
\newtheorem{proposition}[theorem]{Proposition}

\newtheorem{lemma}[theorem]{Lemma}

\theoremstyle{definition}
\newtheorem{definition}[theorem]{Definition}

\newtheorem{remark}[theorem]{Remark}

\newcommand{\N}{{\mathbb N}}
\newcommand{\Z}{{\mathbb Z}}

\DeclareMathOperator{\Sym}{Sym}
\DeclareMathOperator{\Alt}{Alt}

\begin{document}

\onehalfspace

\begin{abstract}
We provide new bounds for the divisibility function of the free group
${\mathbf F}_2$ and construct short laws for the symmetric groups
$\Sym(n)$. The construction is random and relies on
the classification of the finite simple groups. We also give bounds on the length of laws for finite
simple groups of Lie type. 
\end{abstract}

\maketitle

\tableofcontents

\hypersetup{linkcolor=aablue}

\section{Introduction}
\label{intro}
We want to start out by explaining a straightforward and well-known application of the prime number theorem.
We consider the Chebychev functions
$$\vartheta(x) = \sum_{p \leq x} \log(p) \quad \mbox{and} \quad \psi(x) = \sum_{p^k \leq x} \log(p^k),$$
where the summation is over all primes (resp.\ prime powers) less than or equal to $x$. It is well-known and equivalent to the prime number theorem, that
\begin{equation} \label{pnt}
\lim_{x \to \infty}\frac{\vartheta(x)}{x} = 1 \quad \mbox{and} \quad\lim_{x \to \infty}\frac{\psi(x)}{x} = 1.
\end{equation}
We conclude that for each $\varepsilon>0$, there exists $N(\varepsilon) \in \N$ such that for all $n \geq N(\varepsilon)$, there exists some prime $p \leq (1+ \varepsilon)\log(n)$ such that $n$ is not congruent to zero modulo $p$. Indeed, if this was not the case, then
$$n \geq \prod_{p \leq (1+\varepsilon)\log(n)} p = \exp(\vartheta((1+ \varepsilon)\log(n)))$$
for infinitely many $n$, which contradicts Equation \eqref{pnt}. On the other side, setting $n:= \exp(\psi(x))$, we obtain an integer which is congruent to zero modulo any prime power less than or equal $x$. Thus, any prime power $p^k$ so that $n$ is not congruent to zero modulo $p^k$ must satisfy $p^k \geq x \geq (1-\varepsilon)\log(n),$ if $x$ is large enough.

\vspace{.2cm}

These arguments allow to determine the asymptotics of the so-called divisibility function for the group $\Z$, defined more generally as follows.
\begin{definition}For any group $\Gamma$, we define the divisibility function $D_{\Gamma} \colon \Gamma \to \N$ by
$$D_{\Gamma}(\gamma):= \min\{ [\Gamma:\Lambda] \mid \Lambda \subset \Gamma \mbox{ a subgroup}, \gamma \not \in \Lambda \} \in \N \cup \{\infty\}.$$
\end{definition} 
The simplest case to study is $\Gamma=\Z$ and there the prime number theorem implies (as explained above) that 
$$\limsup_{n \to \infty} \frac{D_{\Z}(n)}{\log(n)} = 1.$$
This result appeared in work of Bou-Rabee \cite[Theorem 2.2]{br}, where the study of quantitative aspects of residual finiteness was initiated.
In order to study similar asymptotics for more complicated groups, we restrict to the case that $\Gamma$ is generated by some finite symmetric subset $S \subset \Gamma$. We denote the associated word length function by $|\cdot| \colon \Gamma \to \N$ and set
$$D_{\Gamma}(n) := \max\{ D_{\Gamma}(\gamma) \mid \gamma \in \Gamma, |\gamma| \leq n \}.$$

We now consider the free group ${\mathbf F}_2$ with two generators and
its natural generating set $S:= \{a,a^{-1},b,b^{-1}\}.$ It is
elementary (using Lemma \ref{elem} below) to see that
$D_{{\mathbf F}_2}(n) \leq n+1.$ The only known improvement to this trivial bound has been obtained by Buskin (see \cite{MR2583614})
$$D_{{\mathbf F}_2}(n) \leq \frac n2+2.$$

Concerning lower bounds, it is easy to see the element
$a^{\exp(\psi(n))} \in {\mathbf F}_2$ lies in all subgroups of index
less than or equal $n$ (see again Lemma \ref{elem}), and hence Equation \eqref{pnt} implies 
$$\limsup_{n \to \infty} \frac{D_{\mathbf F_2}(n)}{\log(n)} \geq 1,$$
i.e., $D_{\mathbf F_2}(n) = \Omega(\log(n))$.
Bogopolski asked if $D_{\mathbf F_2}(n) = O(\log(n))$ holds, see \cite[Problem 15.35]{MR2263886}. This question was answered negatively by Bou-Rabee and McReynolds \cite[Theorem 1.4]{brm}, giving the following explicit lower bound.
\begin{equation} \label{mr}
D_{\mathbf F_2}(n) \ge \frac{\log(n)^2}{C\log \log(n)}.
\end{equation}
(Here and below, $C$ refers to absolute constants, which are positive
and sufficiently large, and may denote different constants in
different formulas.)
Independent work of Gimadeev-Vyalyi \cite{MR2972333} on the length of shortest laws for ${\rm Sym}(n)$ implies that
\begin{equation*}
D_{\mathbf F_2}(n) \ge \left(\frac{\log(n)}{C\log\log(n)} \right)^2,
\end{equation*}
but the connection with Bogopolski's question apparently remained unnoticed. We will improve \eqref{mr} and obtain in Theorem \ref{divthm} a new lower bound as follows:
$$D_{{\mathbf F}_2}(n) \ge \exp\left(\left(\frac{\log(n)}{C\log \log(n)} \right)^{1/4} \right).$$
Assuming Babai's Conjecture on the diameter of Cayley graphs of permutation groups (see the discussion below), we even obtain
$$D_{{\mathbf F}_2}(n) \ge \exp\left(\frac{\log(n)}{C\log \log(n)}  \right) = n^{\frac1{C \log\log(n)}}.$$
Note that this estimate is now getting much closer to the linear upper bound for the divisibility function. 

Currently, the proofs are based on consequences of the classification of finite simple groups and it would be desirable to obtain more elementary arguments for these bounds.
It remains to be an outstanding open problem to determine the precise asymptotic behaviour of the function $D_{{\mathbf F}_2}$.

We now turn to the closely related study of two-variable laws for specific groups. 
\begin{definition}
For a group $\Gamma$, we say that $\Gamma$ satisfies the law $w \in {\mathbf F}_2$ if 
$$w(g,h)=e \quad \forall g,h \in \Gamma.$$ 
\end{definition}
We denote by $\Sym(n)$ the symmetric group on the letters
$\{1,\dots,n\}$.  Along the way we will show the existence of short
laws which are satisfied by the group $\Sym(n)$. We denote by
$\alpha(n)$ the length of the shortest law for the group $\Sym(n)$. The shortest previously known laws for $\Sym(n)$
implied 
\begin{equation} \label{oldbound}
\alpha(n)\le \exp\left(C(n \log n)^{1/2} \right),
\end{equation}
a result which was obtained in the course of the proof of \cite[Theorem 1.4]{brm}. The proof of this result was based on the fact that the maximal order of an element in ${\rm Sym}(n)$ is bounded by $\exp(C (n \log(n))^{1/2})$  --- a result of Landau \cite{landau} from 1903. We will give a quick proof of \eqref{oldbound} in Section \ref{divlaw}.
Independently, Gimadeev-Vyalyi \cite{MR2972333} obtained the slightly weaker bound
$$\alpha(n) \le \exp\left(Cn^{1/2} \log(n) \right).$$
We are able to improve these bounds substantially and will show
\begin{theorem}\label{random} 
The length of the shortest non-trivial law for ${\rm Sym}(n)$ satisfies\begin{equation} \label{randomnew}
\alpha(n)\le\exp\left(C\log(n)^4 \log\log(n)\right).
\end{equation}
\end{theorem} The proof uses 
\begin{itemize}
\item a remarkable structure result on the possible primitive subgroups of ${\rm Sym}(n)$ by Liebeck \cite{MR758332} which is ultimately based on the O'Nan-Scott Theorem and the Classification of Finite Simple Groups (CFSG), and
\item a recent breakthrough result by Helfgott-Seress
  \cite{seresshelfgott} on the diameter of Cayley graphs of ${\rm
    Sym}(n)$ (also using (CFSG)).
\end{itemize}
The main contribution to our estimate for the word length, $\exp\left(C\log(n)^4\log\log n)\right)$ comes
from the Helfgott-Seress theorem. The Helfgott-Seress theorem is not
conjectured to be sharp. In fact, a conjecture of Babai
\cite[Conjecture 1.7]{babser} states that this term should be
$n^C$. It is natural to ask how our result would improve under this
conjecture.
\begin{theorem}\label{thm:underBabai}If Babai's conjecture holds, then 
$$\alpha(n)\le\exp\left(C\log(n) \log\log(n)\right) =
  n^{C\log\log(n)}.$$ 
\end{theorem}
The main obstruction to replacing the term $n^{\log\log n}$ with a
polynomial is not the (CFSG) part, for which the existing estimates
are far better than our needs, but in our handling of the transitive
imprimitive subgroups of ${\rm Sym}(n)$, which are wreath products of smaller
groups. See the proofs in Section \ref{sec:proof}.

\vspace{0.2cm}

The complexity of identities for associative rings has a long history and one of the outstanding results is the Amitsur-Levitzki theorem \cite{al} which says that the algebra $M_n(k)$ of $n\times n$-matrices of a field $k$ satisfies the identity
$$\sum_{\sigma \in {\rm Sym}(2n)} {\rm sgn}(\sigma) \cdot A_{\sigma(1)} \cdots A_{\sigma(2n)}=0,$$
for all $2n$-tuples of matrices $A_1,\dots,A_{2n} \in M_n(k)$. Moreover, the theorem asserts that this identity has the minimal degree among all non-trivial identities for $M_n(k)$. Thus, in this case the exact asymptotics of the smallest degree of non-trivial identities is known. It seems that we have still a long way to go to close the current gap between the lower and upper bound and determine the precise asymptotics of $\alpha(n)$.

\vspace{0.2cm}

Similar questions about the length of laws for groups were studied by Hadad \cite{MR2764921} for simple groups of Lie type and
 by Kassabov and Matucci in \cite{MR2784792} for arbitrary finite groups. The lengths of the shortest laws that hold in $n$-step solvable (resp.\ $n$-step nilpotent)  groups were estimated in \cite{thomelk}. Quantitative aspects of almost laws for compact groups (specifically for ${\rm SU}(n)$) were studied in \cites{convergent, thomelk}.

\vspace{0.2cm}

In Section \ref{lietype} we study the length of shortest laws in finite quasi-simple groups of Lie type. We point out a gap in a proof in \cite{MR2764921} and prove a weaker form of the upper bounds in Theorem 1 and Theorem 2 of \cite{MR2764921}.
Our main result asserts that any finite simple group $G$ of Lie type with rank $r$ and defined over a field with $q$ elements satisfies a non-trivial law $w_G \in {\mathbf F}_2$ of length $|w_G| \leq 12 \cdot q^{155\cdot r}$. For $G={\rm PGL}_n(q)$ with $q=p^s$ for some prime $p$, we get more precisely:
$$|w_G| \leq  12 \cdot \exp\left(2\sqrt{2} \cdot n^{1/2} \log(n)\right) \cdot q^{n-1}.$$

It remains to be an intruiging question to determine the asymptotics of the length of the shortest law in ${\rm PGL}_n(p)$ as $p$ increases, even for $n$ fixed or just in the case $n=3$. It is conceivable that our strategy for ${\rm Sym}(n)$ can be carried out for other families of finite simple groups, resulting in improved bounds -- see the remarks at the end of Section \ref{lietype}.

\section{Divisibility and group laws for \texorpdfstring{${\rm Sym}(n)$}{Sym(n)}}
\label{divlaw}
There is a direct relationship between the divisibility function on the free group and the existence of group laws on symmetric groups. 

\begin{lemma} \label{elem}
An element $w \in {\mathbf F}_2$ is a law for ${\rm Sym}(n)$ if and only if $D_{{\mathbf F}_2}(w)>n$.
\end{lemma}
\begin{proof}
Let $w$ be a law for ${\rm Sym}(n)$. Consider a subgroup $\Lambda \subset {\mathbf F}_2$ of index less than or equal $n$. We obtain a natural permutation action of ${\mathbf F}_2$ on the space of cosets ${\mathbf F}_2/\Lambda$ and $|{\mathbf F}_2/\Lambda| \leq n$. By assumption, $w$ fixes every point in ${\mathbf F}_2/\Lambda$. In particular, $w \Lambda = \Lambda$ and hence $w \in \Lambda$. Suppose now that $w$ is not a law for ${\rm Sym}(n)$. Then, there exists permutations $\sigma,\tau \in {\rm Sym}(n)$ and $w(\sigma,\tau) \neq 1_n$. Let us take some point $k \in \{1,\dots,n\}$ such that $w(\sigma,\tau)(k) \neq k$. We set 
$\Lambda := \{v \in {\mathbf F}_2 \mid v(\sigma,\tau)(k)=k \}.$
Then, $\Lambda \subset {\mathbf F}_2$ is a subgroup of index less than or equal $n$ and $w \not \in \Lambda$. In particular, we get $D_{{\mathbf F}_2}(w) \leq n$. This finishes the proof.
\end{proof}

In order to establish estimates like the one in \eqref{mr} or \eqref{randomnew}, it remains to construct laws for ${\rm Sym}(n)$ whose length is as short as possible.

The most common strategy to construct laws for $\Gamma$ is to construct first a list of words $w_1,\dots,w_k$ in ${\mathbf F}_2$ so that for each pair $(g,h) \in \Gamma^2$, we have $w_i(g,h)=e$ for at least one $1 \leq i \leq k$. After this is done, the words are combined using an iterated nested commutator. The following lemma takes care of the second step, similar constructions appeared in \cite[Lemma 3.3]{MR2764921} and \cite[Lemma 10]{MR2784792}.

\begin{lemma}\label{lem:commut}
Let $v_1,\dotsc,v_m$ be non-trivial words in $\mathbf{F}_k$, $k\ge 2$. Then
there exists a non-trivial word $w\in\mathbf{F}_k$ that trivializes
the union of $k$-tuples trivialized by the $v_i$. 
In other words, for every $i\in\{1,\dotsc,m\}$ and $\sigma_1,\cdots,\sigma_k$
in some group, $v_i(\sigma_1,\dots,\sigma_k)=1\Rightarrow w(\sigma_1,\dots,\sigma_k)=1$.
Moreover, the length of $w$ is
bounded by $16\cdot m^2\max|v_i|$.
\end{lemma}
\begin{proof}
We strengthen the requirements on $w$ to $|w|\le
4^{\lfloor\log_2(m-1)\rfloor+2}\max|v_i|$ (for $m\ge 2$, and
$4\max|v_i|$ for $m=1$) and we require that $w$ be a non-trivial
commutator. We now argue by induction on $m$. For $m=1$ we take
$w=[v_1,z]$ for some generator of $z$ of $\mathbf{F}_k$. Since a
commutator is trivial in the free group only when the two words are
powers of the same word \cite{mks}, we choose $z$ to be some
generator such that $v_1$ is not a power of $z$ and we are done.

For $m\ge 2$, we apply the lemma inductively on $v_1,\dotsc,v_{\lfloor
  m/2\rfloor}$ and on $v_{\lfloor m/2\rfloor +1},\dotsc,\linebreak[4]v_m$ and get
two non-trivial words $w_1$ and $w_2$ which together trivialize the
union of the $k$-tuples trivialized by the $v_i$, and
\[
|w_j|\le 4^{\lfloor \log_2(m-1)\rfloor+1}\max|v_i|\qquad \forall j\in\{1,2\}.
\]
Clearly $w=[w_1,w_2]$ trivializes the union of the $k$-tuples trivialized
by the $v_i$ and has the correct length. So we need only show that it
is not trivial. As above, $w$ can be trivial only if $w_1$ and $w_2$
are powers of the same word. But by induction $w_i$ are commutators
and hence cannot be non-trivial powers -- a result that goes back to Sch\"utzenberger \cite{MR0103219}. Hence $w$ can
be trivial only if $w_1=w_2^{\pm 1}$, but in this case $w_1$ and $w_2$
trivialize the same set of $k$-tuples and we can simply take
$w=w_1$ instead. This proves the lemma.
\end{proof}

Let us illustrate this lemma by giving a quick argument for
\eqref{oldbound}. Following Landau \cite{landau}, we denote the
largest order of an element in ${\rm Sym}(n)$ by $g(n)$. Hence the
words $a,a^2,\dotsc,a^{g(n)}\in\mathbf{F}_2$ trivialize all couples in $\Sym(n)$. (Of
course, these words are also in $\mathbf{F}_1$ but Lemma \ref{lem:commut} only
works in $\mathbf{F}_k$, $k\ge 2$ so we add a new variable.) We now apply
Lemma \ref{lem:commut} and get a non-trivial word $w\in\mathbf{F}_2$
with $|w|\le 16g(n)^3$ which trivializes $\Sym(n)$. 
Landau \cite{landau} showed that $g(n) \le \exp\left(C(n \log(n))^{1/2} \right)$ using the Prime Number Theorem. Hence, we can conclude that the length of the shortest law for ${\rm Sym}(n)$ satisfies \eqref{oldbound}.
\vspace{0.1cm}

Let us remark that this last argument demonstrates that the
requirement $k\ge 2$ in Lemma \ref{lem:commut} is really necessary:
even though all our words are in $\mathbf{F}_1$, any word in
$\mathbf{F}_1$ which trivializes all of $\Sym(n)$ must be $a^q$ where
$q$ is divisable by all primes smaller than $n$, which can be shown to
be much larger than \eqref{oldbound}, in fact exponential in $n$,
again by the Prime Number Theorem -- see Equation \eqref{pnt}.

We will use Lemma \ref{lem:commut} repeatedly in a similar way in the proof of Theorem \ref{random}.

\section{Random walks and mixing time}

In this section we recall some standard facts about the relationship between diameter, spectral gap, and mixing time for Cayley graphs of finite groups.

Let $G$ be a finite group and let $S$ be a finite and symmetric generating set. Recall, a set $S\subset G$ is called symmetric, if $S^{-1}:= \{s^{-1} \mid s\in S\}$ is equal to $S$. We consider the lazy random walk on $G$, which starts at the neutral element and whose transitions are defined as follows: $${\mathbb P}(\omega_{n+1} = g \mid \omega_n = h ) = \begin{cases} \frac12 & g=h \\
\frac1{2|S|} & g=sh \mbox{ for some $s \in S$}\end{cases}.$$
Consider the Hilbert space $\ell^2(G)$ with orthonormal basis $\{\delta_g \mid g \in G\}$ and the left-regular representation
$$\lambda \colon G \to {\rm U}(\ell^2(G)), \quad \lambda(g)\delta_h := \delta_{gh}.$$
We set
$$M_S := \frac{1}{2} + \frac1{2|S|}\sum_{g \in S} \lambda(s) \in B(\ell^2(G)),$$
and note that for any subset $E \subset G$,
$${\mathbb P}(\omega_n \in E) = \langle M_S^n(\delta_e),\chi_E \rangle,$$
where $\chi_E$ denotes the characteristic function of $E \subset G$. The properties of the random walk can be studied using the eigenvalues of the self-adjoint operator $M_S$, which we denote by
$$1 = \lambda_0 > \lambda_1(G,S) \geq \lambda_2(G,S) \cdots
\lambda_{|G|}(G,S)\geq 0,$$ 
where the last inequality is due to laziness. Note that the eigenspace for the eigenvalue $1$ is one-dimensional, since the Cayley graph is connected. The difference $1 - \lambda_1(G,S)$ is called the spectral gap of the random walk.
We denote the diameter of the Cayley graph associated with $S$ by 
$${\rm diam}(G,S) := \min\{n \in \N \mid \forall g \in G, \exists s_1,\dots,s_n \in S \cup \{e\} : g=s_1\cdots s_n \}.$$

We have the following relationship between the diameter of a graph and the spectral gap of the random walk, see \cite[Corollary 1]{MR1245303}.

\begin{equation} \label{specgap}
1 -\lambda_1(G,S) \geq  \frac{1}{{2|S| \cdot \rm diam}(G,S)^2}.
\end{equation}

We denote by $u \in \ell^2(G)$ the uniform distribution and note that \eqref{specgap} implies:
$$\left\| M^n(\delta_e) - u \right\|_2 \leq \lambda_1(G,S)^n \leq \left(1 - \frac{1}{2|S| \cdot {\rm diam}(G,S)^2}\right)^n.$$

The following proposition is well-known and uses this information about the speed of convergence to the uniform distribution to relate the diameter with the mixing time of the random walk.

\begin{proposition} \label{mixing}
Let $E \subset G$ be a subset and set $\alpha := |E|/|G|$. 
If $$n \geq 2|S| \cdot{\rm diam}(G,S)^2 \cdot \log\left(2|G| \right),$$
then ${\mathbb P}(\omega_n \in E) \geq \alpha/2.$
\end{proposition}
\begin{proof} Using our assumption, we can compute
$$\left(1 - \frac{1}{2|S| \cdot{\rm diam}(G,S)^2} \right)^n \leq \exp\left(- \frac{n}{2|S| \cdot{\rm diam}(G,S)^2}\right) \leq \frac{1}{2|G|}.$$
and we find the following estimate.
\begin{eqnarray*}
{\mathbb P}(\omega_n \in E) &=&\langle M^n(\delta_e),\chi_E \rangle \\
&\geq& \langle u,\chi_E \rangle -\left(1 - \frac{1}{2|S| \cdot{\rm diam}(G,S)^2}\right)^n \cdot \|\chi_E\| \\
&=& \frac{|E|}{|G|}  -  \left(1 - \frac{1}{2|S| \cdot{\rm diam}(G,S)^2}\right)^n \cdot |E|^{1/2} \\
&\geq& \frac{|E|}{|G|} - \frac{|E|^{1/2}}{2|G|}\\
& \geq& \alpha/2.
\end{eqnarray*}
This finishes the proof.
\end{proof}

\section{Proof of the main result}\label{sec:proof}

We wish to demonstrate Theorems \ref{random} and
\ref{thm:underBabai}. Recall that Theorem \ref{random} states that
%
there exists an absolute constant $C$ such that for each $n \ge 2$, there exists a non-trivial word $w_n \in {\mathbf F}_2$ of length
$$|w_n| \leq \exp\left(C \log(n)^4 \log \log(n) \right),$$
which is trivial whenever evaluated on ${\rm Sym}(n)$.
Theorem \ref{thm:underBabai} gets a better estimate under Babai's
conjecture.
The proof relies on a recent result of Helfgott-Seress \cite{seresshelfgott}, which gives new bound on the diameter of Cayley graphs of ${\rm Sym}(n)$, and a structure result about subgroups of ${\rm Sym}(n)$ obtained by Liebeck \cite{MR758332}. 

Let us first explain this result, which will allow us the perform a crucial induction step in our proof.

\begin{theorem}[Liebeck]\label{maroti}
Let $X$ be a finite set and let $\Gamma \subset {\rm Sym}(X)$ be a subgroup. Then, either
\begin{enumerate}
\item[i)] $\Gamma = {\rm Sym}(X)$ or $\Gamma={\rm Alt}(X)$,
\item[ii)] There exists two disjoint non-empty sets $Y$ and $Z$ such
  that $X\cong Y\cup Z$ and
  $\Gamma\hookrightarrow\Sym(Y)\times\Sym(Z)$.
\item[iii)] There exist two sets $Y$ and $Z$, each with at least two elements, such
  that $X\cong Y\times Z$ and $\Gamma\hookrightarrow
  \Sym(Y)\wr\Sym(Z)$ where the action of $\Sym(Z)$ on $\Sym(Y)^Z$
  implicit in the $\wr$ notation is the natural action.
\item[iv)] There exist two sets $Y$ and $Z$ as above and a $1 \leq k
  \leq |Y|$, such that $X \cong \binom{Y}{k}^Z$ such that $\Gamma$ is
  isomorphic to a group $\Gamma'$ with 
$${\rm Alt}(Y)^Z \subseteq \Gamma'  \subseteq {\rm Sym}(Y) \wr {\rm Sym}(Z),$$
\item[v)] the size of $\Gamma$ is bounded by $$|\Gamma| \leq \exp\left(C_1 \log(n)^2\right).$$
\end{enumerate}
\end{theorem}

Case ii) corresponds to the non-transitive case and every element of
$\Gamma$ preserves the sets $Y$ and $Z$. Case iii) corresponds to the
transitive imprimitive case and each element of $\Gamma$ preserves the
division of $X$ into copies of $Y$. To explain case iv), denote an
element of $X$ by $(y_{i,j})$ where $i\in \{1,\dotsc,|Z|\}$ and $j\in\{1,\dotsc,k\}$.
Then an element $(\sigma_1,\dotsc,\sigma_{|Z|})$ of $\Sym(Y)^Z$ acts on $X$
by $(y_{i,j})\mapsto (y_{i,\sigma_i(j)})$
while the elements of $\Sym(Z)$ act on the $i$ coordinates.
%

\begin{remark}Let us mention that the theorem above has been optimized my Mar\'oti in \cite{MR1943938}, giving a constant in clause v) equal to $4$
with four exceptions, the Matthieu groups $M_{11},M_{12}, M_{23}$
and $M_{24}$.
\end{remark}
Another version of Liebeck's result is the following, achieved in
\cites{maxorder, maxorder2} and \cite{maxorder3}, making heavy use of
CFSG. This version is necessary for our Theorem \ref{thm:underBabai}
and could have been used to simplify the proof of our main Theorem
\ref{random}. Since its proof is not fully published at this time, we
will prove Theorem \ref{random} using Theorem \ref{maroti} and use
Theorem \ref{newthm} only for our Theorem \ref{thm:underBabai}.
To state the theorem, let us say that a permutation $\sigma \in {\rm
  Sym}(n)$ has a {\it regular} cycle, if it admits a cycle of length
equal to its order.

\begin{theorem} \label{newthm}
Let $X$ be a finite set and $G \subset {\rm Sym}(X)$ be a primitive permutation group. Then, either 
\begin{enumerate}
\item[i)] there exists finite sets $Y,Z$, $1 \leq k \leq |Y|$, and an isomorphism
$X \cong \binom{Y}{k}^Z$ such that $${\rm Alt}(Y)^Z \subseteq \Gamma \subseteq {\rm Sym}(Y) \wr {\rm Sym}(Z),$$ or
\item[ii)] every element in $G$ contains a regular cycle.
\end{enumerate}
\end{theorem}

\begin{remark}
Note that the first case in Theorem \ref{newthm} includes the case $k=|Z|=1$, which corresponds to case i) in Theorem \ref{maroti}. The only consequence of Theorem \ref{newthm} that we are going to use is that in Case ii), the order of all elements is bounded by $|X|$.
\end{remark}

Let us now start with the proof of Theorem \ref{random}.
\begin{proof}
As the result is asymptotic we will assume $n$ is sufficiently large. We first want to construct a non-trivial word $v_n \in {\mathbf F}_2$, such that $v_n(\sigma,\tau)=1_n$ whenever the group generated by $\sigma,\tau$ is isomorphic to ${\rm Sym}(k)$ or ${\rm Alt}(k)$ for some $k \leq n$. For simplicity, let us restrict to the case ${\rm Sym}(k)$. We denote by $P(k)$ the set of $k$-cycles in ${\rm Sym}(k)$ and note that $|P(k)|/|{\rm Sym}(k)| = 1/k$.


Let $(\sigma,\tau) \in {\rm Sym}(k)^2$ be a pair of generators and let us assume that $k \leq n$. By the main result in Helfgott-Seress \cite{seresshelfgott}, the diameter of the Cayley graph associated with the set $\{\sigma,\tau,\sigma^{-1},\tau^{-1} \}$ is at most 
$$\exp\left(C  \log(k)^4 \log\log(k) \right),$$
for some universal constant $C>0$.
Thus by Proposition \ref{mixing}, there exists $C>0$, such that a randomly choosen word (according to the standard lazy random walk on ${\mathbf F}_2$) at time 
$$4 \cdot \exp(C  \log(n)^4 \log\log(n))^2  \cdot \log\left(2|{\rm Sym}(n)| \right)$$ will lie in the set $P(k)$ with probability at least $1/2n$.
Now, if we take $8 n^2 \log(n)$ words at the same time independently at random, then the event that none of these words satisfies $w(\sigma,\tau) \in P(k)$ has probability at most
$$(1-1/{2n})^{8n^2 \log(n)} \leq \exp({-4n \ln(n)}).$$

We now consider all $k$ with $k \leq n$ at the same time and note that
there are at most $(n!)^2 \leq \exp(2 n \log n)$ pairs of permutations
that generate ${\rm Sym}(k)$ for some $k \leq n$. Hence, the probability that there exists $k$ with $k \leq n$ and a pair $(\sigma,\tau) \in {\rm Sym}(k)$ that generates ${\rm Sym}(k)$ so that for all of the $8 n^2 \log(n)$ independently choosen words $w$, we have $w(\sigma,\tau) \not \in P(k)$ is less than  $$\exp(- 4  n \log(n)) \cdot \exp(2 n \log n) = \exp(-2 n \log(n))< 1.$$

Hence, there exists a set $W \subset {\mathbf F}_2$ consisting of at most $8n^2 \log(n)$ words of length at most 
$$4 \cdot \exp(C  \log(n)^4 \log\log(n))^2  \cdot \log\left(2|{\rm Sym}(n)| \right) \leq \exp(C'  \log(n)^4 \log\log(n))$$ such that for all $k$ with $k \leq n$ and all pairs $(\sigma,\tau) \in {\rm Sym}(k)^2$ that generate ${\rm Sym}(k)$ at least one word $w \in W$ satisfies $w(\sigma,\tau) \in P(k)$.
Note also that since $P(k)$ does not contain the trivial permutation and $w(\sigma,\tau) \in P(k)$, we have that the $e \not \in W$, where $e$ denotes the neutral element of ${\mathbf F}_2$.

Now, by construction, the order of any element in $P(k)$ is $k \leq n$. We consider the set
$$W' := \left\{w^k \mid w \in W, 1 \leq k \leq n \right\}$$
which consists only of non-trivial words, since ${\mathbf F}_2$ is torsionfree. For each $k$ with $k \leq n$ and any pair $(\sigma,\tau) \in {\rm Sym}(k)$ that generates ${\rm Sym}(k)$, we have $w(\sigma,\tau)=1_k$ for at least one $w \in W'$.

Moreover, the length of elements in $W'$ is bounded by
$$n \cdot \exp(C  \log(n)^4 \log \log(n)^2) \leq \exp(C'  \log(n)^4 \log \log(n)),$$ 
and the number of elements in $W'$ is bounded by
$8n^3 \log(n).$ By Lemma \ref{lem:commut} there is a word $v$ which
trivializes all couples $\sigma,\tau$ in ${\rm Sym}(k)$ which generate either $\Sym(k)$
or $\Alt(k)$ for some $k\le n$, and
\[
|v|\le 4(8n^3\log n)^2\exp(C\log(n)^4\log\log(n))\le
\exp(C'\log(n)^4\log\log(n)).
\]


To finish the proof we need to deal with the case that $(\sigma,\tau)$
generates an arbitrary subgroup. In order to do so we will make use of
Theorem \ref{maroti}. Let us first construct the candidate word $x \in {\mathbf F}_2$,
and then prove its properties. The construction is as follows. We use
induction to construct non-trivial words $x_i$, for all $i$ with $i<\log_2(n)$, which are laws for
$\Sym(2^i)$ of minimal length. We wish to compose $x_i$ and $x_j$, but there is nothing
to guarantee that the result is non-trivial. We therefore embed
$\mathbf{F}_2=\langle a,b\rangle$ into $\mathbf{F}_4=\langle
a,b,c,d\rangle$ and define $x_i^{(1)}=[x_i,c]$ and $x_i^{(2)}[x_i,d]$,
and define 
\[
y_{i,j}'(a,b,c,d)=x_i(x_j^{(1)}(a,b,c,d),x_j^{(2)}(a,b,c,d)).
\]
$y_{i,j}'$ is clearly non-trivial --- in fact, it is reduced as
written. We embed ${\mathbf F}_4$ back in ${\mathbf F}_2$ via
${\mathbf F}_4 = \langle a, bab^{-1}, b^2ab^{-2}, b^3ab^{-3} \rangle
\subset {\mathbf F}_2$, and thus define
$$y_{i,j}(a,b) := y'_{i,j}(a,bab^{-1},b^2ab^{-2},b^3ab^{-3}),$$
which is non-trivial. 

Finally, let $v'$ be a word of length $\exp(C\log(n)^2)$ which
trivializes any element of order less than or equal $\exp(C_1\log(n)^2)$ where $C_1$
is from clause v) in Theorem \ref{maroti}, constructed using Lemma \ref{lem:commut}. To get our candidate $x$, we apply Lemma
\ref{lem:commut} again, this time to the words
\begin{equation}\label{eq:law_induction}
\{v,v'\}\cup\{y_{i,j}:1\le i, j<\log_2(n), 2^{i+j} < 4n\},
\end{equation}
where $v$ was the word constructed in the previous step. The
resulting word is our candidate $x$. We need to show it is indeed a
law for $\Sym(n)$ and to estimate its length.

We first show that $x$ is a law for ${\rm Sym}(n)$.
Indeed, consider a pair $(\sigma,\tau) \in {\rm Sym}(n)$ and denote
the subgroup that the pair generates by $\Gamma \subseteq {\rm
  Sym}(n)$. According to Theorem \ref{maroti}, there are five cases to
consider. Let us first despose of case ii), the case that $\Gamma$ is
non-transitive. In this case we note that if $x$ trivializes the
restrictions of $\sigma$ and $\tau$ to every orbit of $\Gamma$, then
it trivializes $(\sigma,\tau)$. Hence it is enough to show that
$x(\sigma,\tau)=1_k$ for every $\sigma$ and $\tau$ that generate a
transitive subgroup of $\Sym(k)$ for some $k\le n$. This leaves four
cases to verify.

Case i): $\Gamma = {\rm Sym}(k)$ or $\Gamma={\rm Alt}(k)$. In this case the word $v$ vanishes on $(\sigma,\tau)$. Since $x$ is defined to be an iterated commutator involving $v$, it will vanish as well.

Case iii): The action is transitive and imprimitive. In this case the
pair $(\sigma,\tau)$ generates a subgroup of ${\rm Sym}(l) \wr {\rm
  Sym}(k/l)$ for some $l \neq 1$. Let $i\ge 1$ be the smallest number such
that $2^i\ge l$ and $j\ge 1$ the smallest such that $2^j\ge k/l$. To
show that $i$ and $j$ satisfy the requirements, note that
$l,k/l\le n/2$ and hence $i,j<\log_2(n)$ and that $2^{i+j}< (2l)\cdot (2k/l)=4k\le 4n$ so $y_{i,j}$ is one of the
words participating in the construction of $x$. By construction, the word $y_{i,j}$ vanishes on the pair $(\sigma,\tau)$. Indeed, the wreath product admits an extension
$$1 \to {\rm Sym}(l)^{k/l} \to {\rm Sym}(l) \wr {\rm Sym}(k/l)
\stackrel{\pi}{\to} {\rm Sym}(k/l) \to 1.$$ 
Since $x_j$ is a law of $\Sym(k/l)$ we get that
$x_j^{(1/2)}(\alpha,\beta,\gamma,\delta)$ is in the kernel of the extension from
above for any $\alpha,\beta,\gamma,\delta\in\Sym(l)\wr\Sym(k)$. Since $x_i$ is a law for
$\Sym(k)$, we get that $y_{i,j}'(\alpha,\beta,\gamma,\delta)=1_k$ and hence also
$y_{i,j}(\sigma,\tau)$ and $x(\sigma,\tau)$. 

Case iv): In this case the group $\Gamma$ is also contained in wreath product ${\rm Sym}(r) \wr {\rm Sym}(s)$, with
$n=\binom{r}{q}^s$ for some $1 \leq q < r$ and $s,r \geq 2$.  This
implies that $r \leq \sqrt{n}$ and $s \leq \log_2(n)$. For $n$
sufficiently large we can find $i$ and $j$ as in case iii), and we are
done.


Case v): In this case, the group has cardinality bounded by $\exp(C_1 \log(n)^2)$. By construction of $v'$, it vanished on all elements whose order is at most $\exp(C_1 \log(n)^2)$. Hence, the word $x$ vanishes on $\Gamma$.


\vspace{0.1cm}

This finishes the case study and we conclude that indeed, $x$ is a non-trivial
law for ${\rm Sym}(n)$. It remains to bound its length. The number of
terms in \eqref{eq:law_induction} is, up to constants, $\log(
n)^2$. Their lengths satisfy $|v|\le\exp(C\log(n)^4\log\log(n))$,
$|v'|\le\exp(C\log(n)^2)$ and $|y_{i,j}|\le C\alpha(2^i)\alpha(2^j)$
where $\alpha(n)$ is, as before, the length of shortest identities in
${\mathbf F}_2$ that hold for ${\rm Sym}(n)$. Hence
lemma \ref{lem:commut} gives
\begin{equation}\label{eq:alpha_rec}
\alpha(n)\le |x|\le C\log(n)^4\max\left\{\exp(C\log(n)^4\log\log(n)),\max_{i,j}\alpha(2^i)\alpha(2^j) \right\}.
\end{equation}
The theorem is thus proved, given the elementary analysis of
\eqref{eq:alpha_rec} done in Lemma \ref{lem:anal_rec} below.
\end{proof}

\begin{lemma}\label{lem:anal_rec}
A series $\alpha(n)$ satisfying the recursion \eqref{eq:alpha_rec}
satisfies
\[
\alpha(n)< \exp(C\log(n)^4\log\log(n)).
\]
\end{lemma}
\begin{proof}Because of the powers of $2$ appearing in
  \eqref{eq:alpha_rec}, it is natural to define $m$ by
  $n\in(2^{m-1},2^m]$ and show that 
$\alpha(n)\le\exp(Cm^4\log m)$.
Let $K$ be some parameter. We will show that if $K$ is chosen
appropriately (sufficiently large), the inequality 
$\alpha(n)\le\exp(Km^4\log m)$ carries over by induction over
$m$. Assume therefore it holds inductively and examine the term 
$\max_{i,j}\alpha(2^i)\alpha(2^j)$ appearing in
\eqref{eq:alpha_rec}. Recall that the maximum is taken over $i,j<\log_2
(n)$, i.e., $i,j< m$, and such that $2^{i+j}< 4n$ i.e.\ $i+j\le
m+1$. We now claim that under these restrictions there exists some
$m_0$ such that
\begin{equation}\label{eq:infi}
i^4 + j^4 \le m^4 - m^2 \qquad\forall m\ge m_0.
\end{equation}
To see \eqref{eq:infi}, assume without loss of generality that
$i+j=m+1$ and that $i\ge j$ (so that in particular $i>m/2$) and write
\[
i^4+j^4\le (i+j)^4-4i^3j= (m+1)^4 - 4i^3j \le m^4 + 15m^3 - 4i^3j <
m^4 + m^3(15-j/2).
\]
and see that \eqref{eq:infi} holds whenever $j\ge 32$ (regardless of
$m$). For $j< 32$ we use that $i<m$ to get that \eqref{eq:infi} holds
whenever $m>m_0$ for some $m_0$ sufficiently large. 

Now insert \eqref{eq:infi} and the induction hypothesis into \eqref{eq:alpha_rec} and get
\begin{align*}
\log\alpha(n)&\le C+4\log(m)+\max\{Cm^4\log
(m),\max_{i,j}\log\alpha(2^i)+\log\alpha(2^j)\}\\
& \le C+4\log(m)+\max\{Cm^4\log(m),\max_{i,j}Ki^4\log(i)+Kj^4\log(j)\}\\
&\le C+4\log(m)+\max\{Cm^4\log(m),\max_{i,j}(i^4+j^4)\cdot K\log(m)\}\\
&\stackrel{\smash{\textrm{\eqref{eq:infi}}}}{\le}
 C+4\log(m)+\max\{Cm^4\log(m),K(m^4-m^2)\log(m)\}
\end{align*}
whenever $m$ is sufficiently large. Taking $K$ bigger than the
absolute constants $C$ appearing in the last stage, we see that for
$m>m_0$, we may conclude that $\alpha(n)\le
\exp(Km^4\log(m))$ from our induction hypothesis. Increase $K$, if
necessary, so that $\alpha(n)\le \exp(Km^4\log(m))$ for all $m\le
m_0$. With this value of $K$ our induction works and the lemma is proved.
\end{proof}

\begin{proof}[Proof of Theorem \ref{thm:underBabai}]The proof is very
  similar to that of Theorem \ref{random} so we will be brief. Recall
  how \eqref{eq:alpha_rec} was achieved: the term
  $\exp(C\log(n)^4\log\log(n))$ was the length of the word $v$ (see
  the discussion before \eqref{eq:law_induction}) which was the product of the worst diameter
  for any Cayley graph of $\Sym(n)$ with any generators; and several polynomial factors (an $n\log n$ for the difference
  between the square of the diameter and the mixing time, an $n$ because the
  random words need to be taken to a power, and an $n^6\log(n)^2$ for
  the number of random words needed to ensure all generators are
  accounted for). Under Babai's conjecture $\exp(C\log(n)^4\log\log(n))$ can be replaced by
  $n^C$. Further, there was the length of the word $v'$ which was
  $\exp(C\log(n)^2)$ that came from clause v) in Theorem
  \ref{maroti}. Replacing Theorem \ref{maroti} with Theorem
  \ref{newthm} we see that since all remaining cases have order at most $
  n$, the $\exp(C\log(n)^2)$ can also be replaced by a
  polynomial. All in all we get
\[
\alpha(n)\le C\log(n)^4 \max\left\{n^C,\max_{i,j}\alpha(2^i)\alpha(2^j)\right\}.
\]
A similar analysis will show that $\alpha(n)\le n^{C\log\log(n)}$ holds under
this restriction: the main term this time is not the $n^C$ that appears in
the maximum but the accumulation of the $\log(n)^4$  outside the
maximum, which one gets whenever moving from $n$ to $2n$. This finishes the outline of the proof.
\end{proof}

\section{Lower bounds for the divisibility function}

Lemma \ref{elem} allows us to convert the upper bound on the length of the shortest law for ${\rm Sym}(n)$ into a lower bound for divisibility function for ${\mathbf F}_2$. Hence, another way of interpreting our results in Theorem \ref{random} and Theorem \ref{thm:underBabai} is the following theorem.

\begin{theorem} \label{divthm}
The divisibility function for the free group ${\mathbf F}_2$ satisfies the lower bound
$$D_{{\mathbf F}_2}(n) \ge \exp\left(\left(\frac{\log(n)}{C\log \log(n)} \right)^{1/4} \right)$$
for some constant $C>0$.
Moreover, if Babai's Conjecture holds, then the divisibility function for the free group ${\mathbf F}_2$ satisfies the stronger lower bound
$$D_{{\mathbf F}_2}(n) \ge \exp\left(\frac{\log(n)}{C\log \log(n)}  \right) = n^{\frac{1}{C\log\log(n)}}.$$
\end{theorem}
\begin{proof}
Let $w \in {\mathbf F}_2$ be a shortest law for ${\rm Sym}(n)$ and set
$$k:= \lfloor \exp\left(D\log(n)^4 \log \log(n) \right) \rfloor.$$
We compute
\begin{eqnarray*}
\left(\frac{\log(k)}{\log \log(k)} \right)^{1/4} &\leq& \left(\frac{D\log(n) \log \log(n)}{\log\left( D\log(n)^4 \log \log(n) \right)} \right)^{1/4}\\
&=& \left(\frac{D\log(n)^4 \log \log(n)}{\log(D) + 4 \log \log(n) + \log \log \log(n)} \right)^{1/4}\\
&\leq& D/4 \cdot \log(n).
\end{eqnarray*}
Using the estimate from Theorem \ref{random} and Lemma \ref{elem} we get:
\[
D_{{\mathbf F}_2}(k) \geq  D_{{\mathbf F}_2}(w) > n 
\geq\exp\left( 4/D \cdot \left(\frac{\log(k)}{\log \log(k)} \right)^{1/4} \right).
\]
This proves the first claim. The second claim is proved in a similar way using Theorem \ref{thm:underBabai} instead of Theorem \ref{random}. This finishes the proof.
\end{proof}

Note that (at least assuming Babai's Conjecture) we are getting somewhat close to the linear upper bound for $D_{\mathbf F_2}(n)$, which explains in part why attempts have failed to improve this upper bound substantially.

\section{Finite simple groups of Lie type}
\label{lietype}
We now turn to laws satisfied for finite simple groups of Lie type, see \cite[Prop. 2.3.2]{MR1303592} or \cite[Sec. 13.1]{MR0407163} for a detailed discussion.
The proof of the upper bound on the minimal length of a law for such a
group in \cite[Theorem 2]{MR2764921} contains a gap, as we will now explain. In the proof of \cite[Lemma 3.1]{MR2764921}, it is assumed that a polynomial of degree $n$ has a splitting field of degree less than or equal $n$.  This is not correct and a more detailed analysis is needed to overcome this problem. Even though the statement of \cite[Lemma 3.1]{MR2764921} is correct, the author refers to the proof of this lemma in the proof of \cite[Proposition 3.4]{MR2764921} which is crucial to establish the bounds. We will only arrive at a weaker bound than the one obtained in \cite[Theorem 2]{MR2764921}. Our main result is as follows.

\begin{theorem} \label{finitelie}
Let $G$ be a finite simple group of rank $r$ defined over a field with $q$ elements. There exists a law $w_G$ for $G$ with
$|w_G| \leq 48 \cdot q^{155 \cdot r}.$
For $G={\rm PGL}_n(q)$, with $q=p^s$ for some prime $p$, we have more precisely:
$$|w_G|  \leq  48 \cdot \exp\left(2\sqrt{2} \cdot n^{1/2} \log(n)\right) \cdot q^{n-1}.$$
\end{theorem}

\begin{proof} Let $p$ be a prime and $q=p^s$ for some $s \in \N$.
Let $A \in {\rm GL}_n(q)$. Consider the Jordan decomposition $A=A_s A_u=A_uA_s$, where $A_u$ is unipotent and $A_s$ is semi-simple. The characteristic polynomial $\chi_A$ of $A$ is a polynomial of degree $n$ over ${\mathbb F}_q$. Let $j_1,\dots,j_l$ be the degrees of irreducible factors of $\chi_A$, where each degree is counted only once. We set $k := (q^{j_1}-1) \cdots (q^{j_l}-1)$ and claim that $A_s^k=1_n$. Indeed, the algebra ${\mathbb F_q}[A_s] \subset M_{n}({\mathbb F_q})$ splits as a sum of fields isomorphic to ${\mathbb F}_{q^{j_i}}$, for $1 \leq i \leq l$, and every non-zero element $\alpha \in {\mathbb F}_{q^{j_i}}$ satisfies $\alpha^{q^{j_i}-1}=1$. 
Now,
$$\frac{l^2}{2} \leq 1 + \cdots + l \leq j_1 + \cdots + j_l \leq n$$ and hence $l \leq \sqrt{2n}$.
We obtain that there are at most $n^{\sqrt{2n}} = \exp\left(\sqrt{2} \cdot n^{1/2} \log(n)\right)$ possibilities that have to be taken into account. Note also that for each instance we have
$$k=(q^{j_1}-1) \cdots (q^{j_l}-1) \leq q^n.$$
Hence, we obtain at most $\exp\left(\sqrt{2} \cdot n^{1/2} \log(n)\right)$ exponents $k_1,\dots,k_m \in \N$ of size less than or equal $q^n$, such that for every $A \in {\rm GL}_n(q)$ we have the $A_s^{k_i} = 1_n \in {\rm GL}_n(q)$. Since $A_s$ and $A_u$ commute, we conclude that $A^{k_i}$ is unipotent. For any unipotent matrix $B$, it is easy to see that $B^{p^{\lceil \log_2(n) \rceil}} = 1_n.$ Thus, for any matrix $A \in {\rm GL}_n(q)$, we have
$$A^{p^{\lceil \log_2(n) \rceil}k_i }= 1_n, \quad \mbox{for some $1 \leq i \leq m$}.$$
We can now set $w_i = a^{p^{\lceil \log_2(n) \rceil}k_i}$ and proceed. Using Lemma \ref{lem:commut}, we obtain a word $w \in {\mathbf F}_2$ which is trivial on all pairs $A,B \in {\rm GL}_n(q)$ of length
\begin{align} \label{trivialbound}
|w| &\leq  16\exp\left(2\sqrt{2} \cdot n^{1/2} \log(n)\right) \cdot
\left(2 p^{\lceil \log_2(n) \rceil}q^n +2 \right) \leq 48 \cdot q^{5n}.
\end{align}
There is still room for improvement.
First of all, it is known that no element of ${\rm PGL}_n(q)$ has
order $q^n-1$. Thus, we can divide each $k$ (with more than two
factors) appearing above by $q-1$ and the construction would still work. In any case, we obtain $k \leq q^{n-1}$. At the same time, the factor $p^{\lceil \log_2(n) \rceil}$ can only appear in the presence of multiplicities of eigenvalues, which is obvious when looking a the Jordan normal form of the matrix. Thus, the factor of the form $p^{\lceil \log_2(n) \rceil}$ in not necessary, see \cite[Corollary 2.7]{maxorder} for details, and we obtain
\begin{equation} \label{bound}
|w| \leq  48 \cdot \exp\left(2\sqrt{2} \cdot n^{1/2} \log(n)\right) \cdot q^{n-1}.
\end{equation}

This finishes the proof in the case of ${\rm PGL}_n(q)$. The general case follows with the same arguments as in \cite[Section 4.2]{MR2764921}. 
Indeed, it is well-known that there exists a universal constant $D>0$ such that any finite simple group $G$ of Lie type with rank $r$ defined over a field with $q$ elements embeds either into ${\rm PSL}_{Dr}(q)$ or ${\rm SL}_{Dr}(q)$. Thus, any law for ${\rm GL}_{Dr}(q)$ will be a law for $G$ and we can set $C:=5D $ in combination with the estimate in \eqref{trivialbound}. A case study that was carried out in \cite[Section 4.2]{MR2764921} shows that $D=31$, i.e., $C=155$ is enough. This finishes the proof.
\end{proof}

It seems likely that for fixed $n$, much better upper bounds can be achieved for quasi-simple groups of Lie type along the line of arguments that were used to establish Theorem \ref{random}, using 

\begin{itemize}
\item bounds on the diameter of finite simple groups of Lie type of fixed rank, see the fundamental work of Breuillard-Green-Tao \cite{MR2827010} and Pyber-Szab\'o \cite{pyber},
\item the elementary fact that the density of diagonalizable elements in ${\rm GL}_n(q)$ is at least $\exp(-n \log(n))$ if $q \geq n$, and
\item the work of Larsen-Pink, see \cite{MR2813339}, in order to carry out an induction argument similar to the one in the proof of our Theorem \ref{random}.
\end{itemize}

This will be the topic of further investigation.

\section*{Acknowledgments}

This note was written during the trimester on {\it Random Walks and
  Asymptotic Geometry of Groups} at Institute Henri Poincar\'{e} in
Paris. We are grateful to this institution for its hospitality. We are
grateful to Mark Sapir for interesting remarks -- especially about the
comparison with the study of identities for associative algebras, and
for bringing the work of Gimadeev-Vyalyi \cite{MR2972333} to our
attention. The second author thanks Emmanuel Breuillard and Martin
Kassabov for valuable comments. The first author was supported by the
Israel Science Foundation and the Jesselson Foundation.
The second author was supported by ERC Starting Grant No. 277728.

\end{document}